\documentclass[12pt]{amsart}

\usepackage{amsrefs}
\usepackage{amssymb}
\usepackage{fancyhdr}
\usepackage{bbm}
\usepackage[all]{xy}
\usepackage{xcolor}
\usepackage{hyperref}
\hypersetup{
  colorlinks=true,
  linkcolor=blue,
  citecolor=red
}

\newtheorem{theorem}{Theorem}
\newtheorem*{theorem*}{Theorem}
\newtheorem{lemma}[theorem]{Lemma}

\newtheorem{claim}[theorem]{Claim}

\theoremstyle{definition}

\newtheorem*{definition*}{Definition}
\newtheorem*{lemma*}{Lemma}
\usepackage{graphicx}
\usepackage{pstricks, enumerate, pst-node, pst-text, pst-plot}

\numberwithin{equation}{section}
\numberwithin{theorem}{section}

\newcommand{\N}{\mathbb{N}}

\newcommand{\Z}{\mathbb{Z}}

\DeclareMathOperator{\Aut}{Aut}

\def\cc{{\curvearrowright}}

\begin{document}

\title[]{Normal amenable subgroups of the automorphism group of the
  full shift}

\author{Joshua Frisch, Tomer Schlank and Omer Tamuz}
\address[J.\ Frisch, O.\ Tamuz]{California Institute of Technology.}
\address[T.\ Schlank]{Hebrew University.}


\thanks{This research was partially conducted at Microsoft Research,
  New England.}

\date{\today}

\begin{abstract}
  We show that every normal amenable subgroup of the automorphism
  group of the full shift is contained in its center. This follows
  from the analysis of this group's Furstenberg topological boundary,
  through the construction of a minimal and strongly proximal action.

  We extend this result to higher dimensional full shifts. This also
  provides a new proof of Ryan's Theorem and of the fact that these
  groups contain free groups.
\end{abstract}

\maketitle
\tableofcontents
\section{Introduction}
For $n \geq 2$, let $A = \{0,1,\ldots,n-1\}$ be a finite
alphabet. Equip the countable product $A^\Z$ with the product
topology. Let $\sigma \colon A^\Z \to A^\Z$ be the left shift, and let
$\Aut(A^\Z)$ be the group of homeomorphisms of $A^\Z$ that commute
with the shift $\sigma$. The space $A^\Z$ is called the {\em full
  shift}, and $\Aut(A^\Z)$ is called the {\em automorphism group of
  the full shift}. The elements of this group are known as (invertible)
{\em cellular automata}.

This group has been studied extensively, starting with
Hedlund~\cite{hedlund1969endomorphisms}, who showed that it is
countable, and that in many senses it very large; in particular, it
contains free groups of every rank, and hence it is
non-amenable\footnote{Hedlund attributes most of these results to
  Curtis, Lyndon and Hedlund.}. Ryan~\cite{ryan1972theshift} showed
that its center $Z(\Aut(A^\Z))$ is equal to $\Sigma(A^\Z)$, the group
consisting of the powers of the shift $\sigma$. Boyle, Lind and
Rudolph~\cite{boyle1988automorphism} made further progress, extended
many results to automorphism groups of shifts of finite type, and noted
that $\Aut(2^\Z)$ and $\Aut(4^\Z)$ are not algebraically
isomorphic\footnote{It is not known if $\Aut(2^\Z)$ and $\Aut(3^\Z)$
  are isomorphic.}. 

Our main theorem is a strengthening of Ryan's:
\begin{theorem}
  \label{thm:main-z}
  Every normal amenable subgroup of $\Aut(A^\Z)$ is contained in
  $\Sigma(A^\Z)$.
\end{theorem}

For every group $G$ there exists a maximal normal amenable subgroup
called the {\em amenable radical} (see for
example~\cite{nevo1994boundary}); we denote it by $\sqrt{G}$. Thus
this theorem in fact states that $\sqrt{\Aut(A^\Z)} = \Sigma(A^\Z)$.

Furman~\cite{furman2003minimal} showed that the amenable radical is
the kernel of a group's action on its {\em Furstenberg topological
  boundary}. A {\em topological boundary} of a group $G$ is a compact
$G$-space $X$ such that the continuous $G$ action on $X$ is minimal
and strongly proximal~\cite{furstenberg1963poisson,
  glasner1974topological}. The Furstenberg topological boundary $B(G)$
(or the {\em maximal boundary}) is the universal topological boundary,
in the sense that it admits a $G$-equivariant map to any
$G$-boundary. Since $Z(G) \subseteq \sqrt{G}$, and since the extension
of a faithful action is faithful, it follows that
\begin{theorem}[Furman~\cite{furman2003minimal}]
  \label{thm:furman}
  If a group $G$ has a topological boundary $X$ on which the action
  of $G/Z(G)$ is faithful, then $\sqrt{G}=Z(G)$.
\end{theorem}
To prove our main result, Theorem~\ref{thm:main-z}, we construct a
topological boundary of $\Aut(A^\Z)$ whose kernel is equal to
$\Sigma(A^\Z)$; our theorem then follows from Furman's. To show that
our action is strongly proximal we use Glasner's notion of an {\em
  extremely proximal} action~\cite{glasner1974topological}. We define
these terms precisely in the next section.

One can replace $\Z$ with $\Z^d$ in the discussion above. In this case
the shift $\sigma$ and the group it generates are replaced with the
$d$ shifts which generate $\Sigma(A^{\Z^d})$. The automorphism group
$\Aut(A^{\Z^d})$ is defined to be group of homeomorphisms of
$A^{\Z^d}$ that commutes with $\Sigma(A^{\Z^d})$.

Hochman~\cite{hochman2010automorphism} proves in this setting an
analogue of Ryan's Theorem, namely that
$Z(\Aut(A^{\Z^d})) = \Sigma(A^{\Z^d})$, and in fact shows that the
same holds for the automorphism groups of a large class of
shifts. We likewise strengthen his
theorem, for the case of the full shift.
\begin{theorem}
  \label{thm:main-z-d}
  Every normal amenable subgroup of $\Aut(A^{\Z^d})$ is contained in
  $\Sigma(A^{\Z^d})$.
\end{theorem}

Recently, Cyr and Kra~\cite{cyr2015automorphism} showed that some
subshifts\footnote{A {\em shift} or {\em subshift} of $A^\Z$ is a
  closed, shift-invariant subset of $A^{\Z}$.} of $A^\Z$ with
sub-exponential growth have amenable automorphism groups. Their work
follows a number of papers that show that the automorphism group of
``small'' shifts is indeed ``small''~\cite{coven2015automorphisms,
  cyr2014automorphism, cyr2015automorphism2, donoso2015automorphism,
  salo2014toeplitz, salo2014block}.  A natural question is the
following: for which subshifts of $A^\Z$ does it still hold that the
amenable radical of the automorphism group is equal to its center?

\section{A boundary of $\Aut(A^\Z)$}
In this section we denote $G = \Aut(A^\Z)$ and $\Sigma = \Sigma(A^\Z)$.

Let $A^\Z_p \subset A^\Z$ be the set of configurations (as we shall
refer to elements of $A^\Z$) that have a constant infinite prefix:
\begin{align*}
  A^\Z_p = \{x \in A^\Z \,:\, \exists m \in \Z,a \in A \mbox{
  s.t. } x_k = a \mbox{ for all } k \leq m\}.
\end{align*}
Note that this set is invariant to the $G$-action.

Let $A^\Z_* \subset A^\Z_p$ be the result of the exclusion from
$A^\Z_p$ of the $n$ constant configurations. This set is still
$G$-invariant. Given $x \in A^\Z_*$, let $\ell(x)$ denote the last
coordinate of the constant prefix:
\begin{align*}
  \ell(x) = \min\{m \in Z\,:\, x_m \neq x_{m+1}\}.
\end{align*}
Note that $\ell(\sigma x) = \ell(x) - 1$, in general $\ell(\sigma^k x)
= \ell(x)-k$, and in particular $\ell(\sigma^{\ell(x)}x) = 0$.

Let $\Omega$ be given by
\begin{align*}
  \Omega = \{x \colon \{0,1,2,\ldots\} \to A \,:\, x_0 \neq x_1\}.
\end{align*}
This is the space of one-sided infinite configurations, in which the
zeroth symbol is different than the first. We equip it with the
natural topology induced from the product topology.

To define a $G$ action on $\Omega$, let $\varphi \colon \Omega \to
A^\Z_*$ assign to $x \in \Omega$ the two-sided configuration in
which a one-sided, infinite constant $x_0$ prefix precedes
$x_1x_2\ldots$. Formally:
\begin{align*}
  [\varphi(x)]_m =
  \begin{cases}
    x_0& m \leq 0\\
    x_m&\mbox{otherwise}
  \end{cases}.
\end{align*}
Note that the image of $\varphi$ is all the configurations in
$A^\Z_*$ for which $\ell(x) = 0$. Hence
$\varphi^{-1}(\sigma^{\ell(x)}x)$ is well defined for every $x \in
A^\Z_*$. Accordingly, define $\psi \colon A^\Z_* \to \Omega$ by
$\psi(x) = \varphi^{-1}(\sigma^{\ell(x)}x)$; note that $\psi \circ
\varphi$ is the identity. Now, given $g \in G$ and $x \in \Omega$ let
\begin{align*}
  g x = \psi\left(g\varphi(x)\right). 
\end{align*}
It is straightforward to verify that this is indeed a $G$-action on
$\Omega$. Note also that this action factors through $G/\Sigma$, since
$\sigma f = \sigma$ for all $x \in \Omega$. Additionally, it is easy
to see that the $G/\Sigma$-action is faithful; this is due to the fact
that $\psi^{-1}(\Omega)$ is dense in $A^\Z$.

For $a \in A$, let $\Omega^a = \{x \in \Omega \,:\, x_0=a\}$. Let $G^a
\subset G$ be the finite index subgroup that fixes $\Omega^a$, and let
$G^* = \cap_a G^a$.

\subsection{Extreme proximality}

An action $H \cc X$ of a discrete group on a compact metric space is
said to be {\em extremely proximal} if, for any closed $Y \subsetneq
X$, there exists a sequence $\{h_k\} \subset H$ such that $\lim_k h_k Y$
is a singleton, where the limit is taken in the Hausdorff
topology~\cite{glasner1974topological}.

An action $H \cc X$ of a discrete group on a compact Hausdorff space
is said to be {\em strongly proximal} if, for any Borel probability
measure $\mu$ on $X$, there exists a sequence $\{h_k\} \subset H$ such
that $\lim_k h_k\mu$ is a point mass, where the limit is taken in the
weak* topology~\cite{furstenberg1963poisson, glasner1974topological}.
We prove the following theorem in Section~\ref{sec:cellular}.
\begin{theorem}
  \label{prop:extremely-proximal}
  The action $G^* \cc \Omega^a$ is extremely proximal for all $a \in A$.
\end{theorem}

We can now conclude that $G^*$ is not amenable, and in fact includes a
free group with two generators. This follows from the following
theorem.
\begin{theorem}[Glasner~\cite{glasner1974topological}]
  If a group has a non-trivial minimal extremely proximal action then
  it contains a free subgroup on two generators.
\end{theorem}
The fact that $G^*$ is contains a free subgroup on two generators was
already shown in~\cite{hedlund1969endomorphisms}, and thus this
provides a new proof of that fact.

\subsection{Induction to a strongly proximal, minimal action}

Recall that $\Omega^a$ is the subset of all $x \in \Omega$ such that
$x_0=a$.  Let $\bar\Omega$ be the collection of subsets of $\Omega$
which intersect each $\Omega^a$ in exactly one element. Hence every
element of $\bar\Omega$ is a set of size $n = |A|$, and can be written
as
\begin{align*}
  \{x^0,x^1,\ldots,x^{n-1}\}
\end{align*}
with $x^a_0 = a$ for all $a \in A$. The topology on $\bar\Omega$ is
inherited from $\Omega$ in the obvious way.

There is a natural $G$ action on $\bar\Omega$, derived from the action
on $\Omega$, and hence on the subsets of $\Omega$.

Since the elements of $G^*$ preserve $x_0$ - that is, $[g x]_0 = x_0$
for all $x \in \Omega$ and $g \in G^*$ - the action of $G^*$ on
$\bar\Omega$ is isomorphic to the diagonal action
$G^* \curvearrowright \Omega^0 \times \Omega^1 \times \cdots \times
\Omega^{n-1}$.
Since each action $G^* \curvearrowright \Omega^a$ is extremely
proximal (Theorem~\ref{prop:extremely-proximal}), the product action
$G^* \curvearrowright \Omega^0 \times \Omega^1 \times \cdots \times
\Omega^{n-1}$
is strongly proximal. This follows from the facts that (i) extremely
proximal actions are strongly proximal and (ii) that a product of
strongly proximal actions is likewise strongly
proximal~\cite{glasner1974topological}. Hence we have shown the
following theorem.
\begin{theorem}
  \label{thm:proximal}
  The $G^*$ action on $\bar\Omega$ is strongly proximal.
\end{theorem}

We next show that this action is also minimal.
\begin{theorem}
  \label{thm:minimal}
  The $G^*$ actions on both $\bar\Omega$ and $\Omega^0$ are minimal.
\end{theorem}
We prove this theorem in Section~\ref{sec:cellular}.

\section{The full shift over $\Z$}
Given the construction of the previous section, the proof of our main
theorem is immediate.
\begin{proof}[Proof of Theorem~\ref{thm:main-z}]
  Note that $\bar\Omega$ is a topological boundary of $\Aut(A^\Z)$,
  since it is strongly proximal (Theorem~\ref{thm:proximal}) and
  minimal (Theorem~\ref{thm:minimal}). Since this action is a faithful
  action of $\Aut(A^\Z)/\Sigma(A^\Z)$, the claim follows by
  Theorem~\ref{thm:furman}.
\end{proof}

\section{The full shift over $\Z^d$}

In this section we extend our result to show that the amenable radical
of $\Aut(A^{\Z^d})$ is the group of shifts. We do this by essentially
reducing the higher dimensional case to the one dimensional case.

Fix a dimension $d$. For $k \in \N$, let $M_k$ be the basis for $\Z^d$
given by the rows of the following matrix:
\begin{align*}
  \begin{pmatrix}
    1 &k &0 &\cdots& 0 &0\\
    0 &1 &k &\cdots& 0 &0\\
    \vdots&\vdots&\vdots&\ddots&\vdots&\vdots\\
    0 &0 &0 &\cdots& 1 &k\\
    0 &0 &0 &\cdots& 0 &1
  \end{pmatrix}.
\end{align*}
Let ${\bf v}$ be the unit vector of the $d$\textsuperscript{th}
coordinate (which is also the last vector in $M_k$), and let
$U_k \subset \Z^d$ the span of the first $d-1$ vectors in $M_k$.  Then
every element of $\Z^d$ can be uniquely written as
${\bf u}+\ell \cdot {\bf v}$ where ${\bf u}\in U_k$ and $\ell\in \Z$.

The important property of $U_k$ is the following.
\begin{claim}
  \label{clm:norm}
  Every nonzero element ${\bf u} \in U_k$ has norm greater than $k$.  
\end{claim}
\begin{proof}
  Represent ${\bf u}$ as an integral linear combination of the first
  $d-1$ vectors in $M_k$, and note that for ${\bf u}$ to be nonzero
  there must be a largest index $i$ such that the coefficient given by
  the $i$\textsuperscript{th} basis vector is nonzero.  This implies
  that the $i+1$\textsuperscript{st} coordinate is a nonzero multiple
  of $k$, which implies that the norm of ${\bf u}$ is at least $k$.
\end{proof}

Let $A^{\Z^d}_{U_k}$ be the subset of $A^{\Z^d}$ which is periodic mod
$U_k$; that is, under the natural shift action of $\Z^d$ on
$A^{\Z^d}$, $A^{\Z^d}_{U_k}$ is the set of $U_k$-invariant elements of
$A^{\Z^d}$. Note that $A^{\Z^d}_{U_k}$ is a closed subset of
$A^{\Z^d}$. We endow it with the induced topology. After proving a
simple claim, we will proceed to show how $A^{\Z^d}_{U_k}$ can be
identified with $A^\Z$.

The next claim follows directly from the definition of
$A^{\Z^d}_{U_k}$ and Claim~\ref{clm:norm}.
\begin{claim}
  \label{clm:ball}
  The projection of $A^{\Z^d}_{U_k}$ to $B_k$, the ball of radius $k$
  in $\Z^d$, is equal to $A^{B_k}$.
\end{claim}
That is, any $x \in A^{B_k}$ can be completed to an element of
$A^{\Z^d}_{U_k}$.

Using the obvious group isomorphism between $\Z$ and
$\Z \cdot {\bf v} \subset \Z^d$ (recall that ${\bf v}$ is the last
vector in $M_k$), we obtain the homeomorphism
$\pi \colon A^{\Z^d}_{U_k} \to A^\Z$ given by
\begin{align*}
  [\pi(x)]_n = x_{n \cdot {\bf v}}.
\end{align*}
Note that this is indeed a bijection since $\Z \cdot {\bf v}$ is a set
of representatives $\Z^d / U_k$, the cosets of $U_k$ in $\Z^d$. It is
straightforward to check that $\pi$ is also continuous.

Accordingly, we define a group homomorphism
$\phi_k \colon \Aut(A^{\Z^d}) \to \Aut(A^\Z)$ by setting
$\phi_k(g) = \pi \circ g \circ \pi^{-1}$.  Note that
$\pi^{-1} \circ \sigma \circ \pi$, the conjugation of the shift on
$A^\Z$ by $\pi$, is a shift on $A^{\Z^d}_{U_k}$. It follows that
$\phi_k(g)$ commutes with the shift $\sigma$, and hence the image of
$\phi_k$ is indeed in $\Aut(\Z)$.

Let $L_r \subset \Aut(A^{\Z^d})$ be the set of cellular automata with
memory less than $r$. That is, $g \in L_r$ if $[g(x)]_0$ is determined
by the restriction of $x$ to some ball of radius less than $r$ around
$0$.
\begin{claim}
  \label{clm:A_k}
  If $g \in L_k$ then $g$ is the unique element in
  $\phi_k^{-1}(g) \cap L_k$.
\end{claim}
\begin{proof}
  Every $g \in L_k$ is uniquely determined, among elements of $L_k$,
  by its action on $A^{\Z^d}_{U_k}$; this follows from
  Claim~\ref{clm:ball}. Since the kernel of $\phi_k$ is the kernel of
  the action $\Aut(A^{\Z^d}) \curvearrowright A^{\Z^d}_{U_k}$, it follows
  that if $g$ has memory less than $k$ then $g$ is the unique element
  in $L_k$ that is mapped to $\phi(g)$.
\end{proof}
Note that there may, however, be other elements of $\Aut(A^{\Z^d})$, which
will not be in $L_k$, whose action on $A^{\Z^d}_{U_k}$ is the same as
that of $g$.

We will exploit these homomorphisms $\phi_k$ for varying $k$ in order
to prove our theorem. We first note the following easy lemma.
\begin{lemma}
  \label{lemma:radical}
  Let $\phi: H\to K$ be a group homomorphism, and let $\sqrt H$ denote
  the amenable radical of $H$. Then
  $\sqrt H \subseteq \phi^{-1}(\sqrt K )$.
\end{lemma}
\begin{proof}
  This follows immediately from the fact that both amenable groups and
  normal subgroups are preserved under quotients.
\end{proof}

We are now in a position to prove the main theorem of this section.
\begin{proof}[Proof of Theorem~\ref{thm:main-z-d}]
  Choose $g \in \sqrt{\Aut(A^{\Z^d})}$. Then $g$ has memory less than $k$
  for some large enough $k$, i.e., $g \in L_k$.
  From~\ref{lemma:radical} it follows that
  $\phi_k(g) \in \sqrt{\Aut(\Z)}$, and so $\phi_k(g)$ is a shift
  $\sigma^m$, for some $m \in \Z$, by Theorem~\ref{thm:main-z}.

  We now claim that since $\phi_k(g)$ is a shift and since
  $g \in L_k$, then $g$ is a shift. Showing this will conclude the
  proof that the amenable radical of $\Aut(A^{\Z^d})$ is equal to the
  shifts. To see this, note that since $\phi_k(g) = \sigma^m$ then for
  every $x \in A^\Z$, $[\phi_k(g)x]_0 = x_m$. Hence, by the definition
  of $\phi_k$, for every $y \in A^{\Z^d}_{U_k}$ it holds that
  $[g y]_{\bf 0} = y_{m \cdot{\bf v}}$. By the definition of
  $A^{\Z^d}_{U_k}$, $y_{m \cdot{\bf v}} = y_{m \cdot{\bf v}+{\bf u}}$
  for every ${\bf u} \in U_k$. Since $g \in L_k$, it follows that the
  norm of $m \cdot{\bf v}+{\bf u}$ is at most $k$ for some ${\bf u}$.
  Therefore the shift by $m \cdot{\bf v}+{\bf u}$ is also in
  $L_k$. But by Claim~\ref{clm:A_k} the unique element that is both in
  $L_k$ and $\phi^{-1}(g)$ is $g$, and hence $g$ is the shift by
  $m \cdot{\bf v}+{\bf u}$.
  
\end{proof}

\section{Construction of cellular automata}
\label{sec:cellular}
\subsection{Defining cellular automata}
\label{sec:cellulars-general}
To define invertible cellular automata on $A^\Z$ we will use the
following general scheme\footnote{This scheme is a generalization of
  the one used in~\cite[Section 2]{boyle1988automorphism}.}. First, we
fix a {\em start marker} $S$ and an {\em end marker} $E$, where
$S \in A^k$ and $E \in A^{k'}$ for some $k,k' \in \N$. We choose an
$n \in \N$ and call some subset $\mathcal{D} \subseteq A^n$ the set of
possible {\em data}. Finally, we let $\pi$ be a bijection
$\pi \colon \mathcal{D} \to \mathcal{D}$.  In our constructions, $\pi$
will always be an involution. We then define a cellular automaton
$g \colon A^\Z \to A^\Z$ by the mapping $S D E \to S \pi(D) E$, where
the {\em data} $D$ is an element of $\mathcal{D}$. That is, let
$x \in A^\Z$, and denote $x_{m,m'} = x_m x_{m+1}\ldots x_{m'-1}$.
Then if $x_{m,m+k+n+k'} = S D E$ for some $D \in \mathcal{D}$ then
$[g(x)]_{m,m+k+n+k'} = S \pi(D) E$, and everywhere else $g$ is the
identity. We will use the following notation to define particular
automata. For example, if $S = 000$, $E = 111$,
$\mathcal{D} = \{2332,3223\}$ and $\pi(2332)=(3223)$ then we will
define $g$ by the diagram below, in which the data appear in boldface.
\begin{align*}
  \xymatrix{ 000{\bf 2332}111 \ar[d]^{g} \\
  000{\bf 3223} 111 \ar[u]}
\end{align*}

For such an automaton to be well defined, it suffices to show that no
two data matches overlap; that is, if $x_{m,m+k+n+k'} = S D E$ and
$x_{m',m'+k+n+k'} = S D' E$ for some $m \neq m'$ and
$D, D' \in \mathcal{D}$, then the data match intervals $[m+k,m+k+n)$
and $[m'+k,m'+k+n)$ do not intersect. To show that such an automaton
is invertible, it suffices to furthermore show that no data match
overlaps a marker match. That is, the data match interval $[m+k,m+k+n)$
does not intersect either of the marker match intervals $[m',m'+k)$
and $[m'+k+n,m'+k+n+k')$. We refer to these conditions below as the
{\em overlap} conditions.

We will need to slightly generalize this construction to cellular
automata where there is a collection of start markers
$(S_1,\ldots,S_\ell)$, corresponding end markers $(E_1,\ldots,E_\ell)$
corresponding data sets $(\mathcal{D}_1,\ldots,\mathcal{D}_\ell)$ and
corresponding bijections $(\pi_1,\ldots,\pi_\ell)$. As before, if
$x_{m,m+k+n+k'} = S_i D E_i$ for some $D \in \mathcal{D}_i$ then
$[g(x)]_{m,m+k+n+k'} = S_i \pi_i(D) E_i$. To ensure well-defindedness
and invertibility, similar overlap conditions need to apply. That is,
that no data match overlaps another data or marker match, whether of
the same index $i$ or not. To specify such automata we will use
similar diagrams, for example
\begin{align*}
  \xymatrix{ 000{\bf 2332}111 \ar[d]^{g} & 1000{\bf 43}00 \ar[d]^{g}
  \\
  000{\bf 3223} 111 \ar[u] & 1000{\bf 34}00 \ar[u]}
\end{align*}
Note that in this example the markers do overlap, but the data cannot.

\subsection{Proof of Theorem~\ref{thm:minimal}}

\begin{proof}[Proof of Theorem~\ref{thm:minimal}]
  We show that the $G^*$ action on $\bar\Omega$ is minimal. The proof
  that the action on $\Omega^0$ is minimal follows by the same
  argument.

  To this end, we choose arbitrary
  $\bar\omega,\bar\eta \in \bar\Omega$ and show that there exists a
  sequence $g_k \in G^*$ such that $\lim_k g_k\bar\omega = \bar\eta$.

  For each $a \in A$, define $x^a, y^a \in \Omega^a$ by
  $\bar\omega \cap \Omega^a = \{x^a\}$ and
  $\bar\eta \cap \Omega^a = \{y^a\}$. Denote
  $x_{[k]}^a = x_1^a\ldots x_k^a \in A^k$ and likewise
  $y_{[k]}^a = y_1^a\ldots y_k^a \in A^k$. 

  Define the transformation $g_k$ as follows.
  \begin{align*}
    \xymatrix{ a^{2^k}{\bf a^k }a x_{[k]}^a \ar[d]^{g_k}
    \\
    a^{2^k}{\bf y_{[k]}^a}a x_{[k]}^a \ar[u]}
    \quad\mbox{ for each }
    a \mbox{ s.t. } x_{[k]}^a \neq y_{[k]}^a
  \end{align*}
  Note that the choice of $2^k$ in the start marker is somewhat
  arbitrary; we could have chosen any function that is sufficiently
  larger than $k$.

  Note also that $x_1^a$ and $y_1^a$ are both not equal to $a$. Using
  this, and that fact that the end marker starts with $a$, it is
  straightforward (if tedious) to verify that the overlap conditions
  of Section~\ref{sec:cellulars-general} are satisfied.
  
  Finally, it is likewise easy to see that $\lim_kg_k x^a = y^a$ in
  particular, $[g_k(x^a)]_{[k]} = y_{[k]}$. Hence
  $\lim_kg_k\bar\omega = \bar\eta$.
  
\end{proof}

\subsection{Proof of Theorem~\ref{prop:extremely-proximal}}
Let $o \in \Omega^0$ be given by $o_0 = 0$ and $o_m = 1$ for all
$m>0$. Given $f \in \Omega^0$, let $r(f)$ measure the length of the
initial sequence of ones in $f$:
\begin{align*}
  r(f) = \min\{m \geq 0\,:\,f_{m+1} \neq 1\}.
\end{align*}
Note that $r(f)$ is well defined for $f \in \Omega^0$ except $o$; we
define $r(o) = \infty$. For $m \in \N$ define
\begin{align*}
  C_m = r^{-1}(m) = \{f \in \Omega^0\,:\,r(f) = m\}.
\end{align*}
Note that $\cup_{m=1}^\infty C_m=\Omega^0 \setminus \{o\}$, that each
$C_m$ is closed, and that $\lim_m C_m = \{o\}$ for all $m$.

We now define a sequence $\{g_k\}_{k > 0} \subset G^*$ as follows:
\begin{align*}
  \xymatrix{ 0^{2^k}{\bf 0^k}1^y0 \ar[d]^{g_k}
  \\
  0^{2^k}{\bf 1^k}1^y0 \ar[u]} \quad\mbox{ for each } 0 < y < k
\end{align*}
If $|A| \geq 3$, then in addition we let $a$ be any symbol in $A$ that
does not equal $0$ or $1$, and add to $g_k$ the following transformations:
\begin{align*}
  \xymatrix{ 0^{2^k}{\bf 0^k}1^ya \ar[d]^{g_k}
  \\
  0^{2^k}{\bf 1^k}1^ya \ar[u]} \quad\mbox{ for each } 0 \leq y < k \mbox{
  and } a \in A \setminus \{0,1\}
\end{align*}
For example, two transformations performed by $g_3$ are
\begin{align*}
  \xymatrix{ 00000000{\bf 000}110 \ar[d]^{g_3}& &00000000{\bf 000}2\ar[d]^{g_3} \\
   00000000{\bf 111}110 \ar[u]& &00000000{\bf 111}2\ar[u]}
\end{align*}

Using the fact that $y$ is strictly less than $k$, it is
straightforward to check the overlap conditions of
Section~\ref{sec:cellulars-general}, and hence each $g_k$ is a well defined
involution.

Now, note that if $r(f) < k$ then $r(g_k f) = r(f) + k$. Hence, for
all $k > m$ we have that $g_k C_m \subseteq C_{m+k}$. Hence
\begin{claim}
  \label{clm:lim-g-k}
  $\lim_kg_k C_m = \lim_k C_{m+k} = \{o\}$.
\end{claim}

\begin{proof}[Proof of Theorem~\ref{prop:extremely-proximal}]
  Without loss of generality, assume $a=0$. Let
  $C \subsetneq \Omega^0$ be closed. Then there exists a $g_0 \in G^*$
  such that $g_0C$ does not include $o$, by the minimality of the
  $G^*$ action on $\Omega^0$ (Theorem~\ref{thm:minimal}). Since $g_0C$
  is closed, it is disjoint from some neighborhood of $o$, and so it
  is contained in the finite union $\cup_{i=1}^m C_i$, for some $m$
  large enough. Let $\bar{g}_k = g_kg_0$, where $g_k$ is as defined
  above. Then
  $$\lim_k\bar{g}_k C  = \lim_kg_kg_0C \subseteq
  \lim_kg_k\cup_{i=1}^m C_i = \cup_{i=1}^m\lim_kg_k C_i =\{o\},$$ where
  the last equality follows from Claim~\ref{clm:lim-g-k}.  But the
  first limit cannot be an empty set, and so
  $$\lim_k\bar{g}_k C = \{o\}.$$
\end{proof}

\bibliography{aut_sigma}
\end{document}